\documentclass[a4paper]{amsart}
\usepackage{color}
\usepackage[latin1]{inputenc}
\usepackage[T1]{fontenc}
\usepackage{amsfonts}
\usepackage{amssymb}
\usepackage{amsmath}
\usepackage{amsthm}
\usepackage{stmaryrd}
\usepackage{enumerate}
\usepackage{multirow}
\usepackage{graphicx}
\usepackage{setspace}
\usepackage{graphicx,type1cm,eso-pic,color}
\usepackage{pstricks}
\usepackage{float}
\newtheorem{theorem}{Theorem}[section]
\newtheorem{lemma}[theorem]{Lemma}
\newtheorem{proposition}[theorem]{Proposition}

\newtheorem{definition}[theorem]{Definition}
\newtheorem{assumption}[theorem]{Assumption}

\newtheorem{example}[theorem]{Example}
\begin{document}
\setlength\arraycolsep{2pt}
\title[On a Calculable Projection Estimator of the Drift Function in Fractional SDE]{On a Calculable Skorokhod's Integral Based Projection Estimator of the Drift Function in Fractional SDE}
\author{Nicolas MARIE$^{\dag}$}
\address{$^{\dag}$Laboratoire Modal'X, Universit\'e Paris Nanterre, Nanterre, France}
\email{nmarie@parisnanterre.fr}
\date{}
\maketitle
\noindent
%


%
\begin{abstract}
This paper deals with a Skorokhod's integral based projection type estimator $\widehat b_m$ of the drift function $b_0$ computed from $N\in\mathbb N^*$ independent copies $X^1,\dots,X^N$ of the solution $X$ of $dX_t = b_0(X_t)dt +\sigma dB_t$, where $B$ is a fractional Brownian motion of Hurst index $H\in (1/2,1)$. Skorokhod's integral based estimators cannot be calculated directly from $X^1,\dots,X^N$, but in this paper an $\mathbb L^2$-error bound is established on a calculable approximation of $\widehat b_m$.
\end{abstract}
\tableofcontents
%


%
\section{Introduction}\label{section_introduction}
Consider the differential equation
\begin{equation}\label{main_equation}
X_t = X_0 +\int_{0}^{t}b_0(X_s)ds +\sigma B_t
\textrm{ $;$ }t\in [0,T],
\end{equation}
where $T > 0$, $B = (B_t)_{t\in\mathbb R}$ is the two-sided fractional Brownian motion of Hurst index $H\in (1/2,1)$ constructed in Hairer and Ohashi \cite{HO07} (Section 4.1), $X_0$ is a square integrable and $\sigma((B_t)_{t\in\mathbb R_-})$-measurable random variable, $b_0\in C^1(\mathbb R)$ and $\sigma\in\mathbb R^*$. Moreover, $b_0'$ is bounded, and then Equation (\ref{main_equation}) has a unique solution $X$.
\\
\\
The oldest kind of estimators of the drift function $b_0$ is based on the long-time behavior of the solution of Equation (\ref{main_equation}). In the parametric estimation framework, the reader may refer to the monograph \cite{KMR17} written by K. Kubilius, Y. Mishura and K. Ralchenko, but also to Kleptsyna and Le Breton \cite{KL01}, Tudor and Viens \cite{TV07}, Hu and Nualart \cite{HN10}, Neuenkirch and Tindel \cite{NT14}, Hu et al. \cite{HNZ19}, Marie and Raynaud de Fitte \cite{MRF21}, etc. In the nonparametric estimation framework, the reader may refer to Saussereau \cite{SAUSSEREAU14} and Comte and Marie \cite{CM19}.
\\
\\
Recently, a new kind of estimators of $b_0$ have been investigated: those computed from $N$ independent copies $X^1,\dots,X^N$ of $X$ observed on $[0,T]$ with $T$ fixed but $N\rightarrow\infty$. The reader may refer to Marie \cite{MARIE23} in the parametric framework, but also to Comte and Marie \cite{CM21} dealing with an extension to Equation (\ref{main_equation}) of the projection least squares estimator computed from independent copies of a diffusion process in Comte and Genon-Catalot \cite{CGC20}.
\\
\\
Note that copies based estimators are well-adapted to some situations difficult to manage with long-time behavior based estimators. For instance, in pharmacokinetics, $X$ may model the elimination process of a drug administered to one people and then, in a clinical-trial involving $N$ patients, $X^i$ may model the elimination process of the same drug for the $i$-th patient.
\\
\\
The stochastic integral involved in the definition of the estimators studied in \cite{HN10}, \cite{HNZ19}, \cite{MRF21}, \cite{CM19}, \cite{CM21} and \cite{MARIE23} is taken in the sense of Skorokhod. To be not directly calculable from one observation of $X$ is the major drawback of the Skorokhod integral with respect to $X$. To bypass this difficulty by extending the fixed point strategy of Marie \cite{MARIE23} to the nonparametric estimation framework is the main purpose of this paper.
\\
\\
Consider the copies $X^1,\dots,X^N$ of $X$ such that
\begin{displaymath}
X^i :=\mathcal I_T(X_{0}^{i},B^i)
\textrm{ $;$ }
\forall i\in\{1,\dots,N\},
\end{displaymath}
where $\mathcal I_T$ is the solution map for Equation (\ref{main_equation}) and $B^1,\dots,B^N$ (resp. $X_{0}^{1},\dots,X_{0}^{N}$) are independent copies of $B$ (resp. $X_0$). Moreover, assume that the probability distribution of $X_t$, $t\in (0,T]$, has a density $f_t$ with respect to Lebesgue's measure such that $t\mapsto f_t(x)$ belongs to $\mathbb L^1([0,T])$ for every $x\in\mathbb R$. This legitimates to consider the density function $f$ defined by
\begin{displaymath}
f(x) :=\frac{1}{T}\int_{0}^{T}f_s(x)ds
\textrm{ $;$ }
\forall x\in\mathbb R.
\end{displaymath}
Consider the interval $I$ of $\mathbb R$ on which $x\mapsto b_0(x)$ will be estimated, and let $(\varphi_1,\dots,\varphi_m)$ be an $I$-supported orthonormal family of $\mathbb L^2(\mathbb R,dx)$. For instance, if $I$ is compact, one can take the trigonometric basis, and if $I =\mathbb R$, one can take the Hermite basis. A suitable nonparametric estimator of $b_0$ could be
\begin{displaymath}
\widehat b_m(x) :=
\frac{1}{f(x)}\sum_{j = 1}^{m}\left(\frac{1}{NT}
\sum_{i = 1}^{N}\int_{0}^{T}\varphi_j(X_{s}^{i})\delta X_{s}^{i}\right)\varphi_j(x)
\textrm{ $;$ }x\in I,
\end{displaymath}
but it is not directly calculable from $X^1,\dots,X^N$ because the stochastic integral is taken in the sense of Skorokhod. However, by assuming that $\varphi_1,\dots,\varphi_m$ are differentiable on $I$,
\begin{equation}\label{relation_pLS_drift}
\widehat b_m =\Phi_m(b_0)\approx\Phi_m(\widehat b_m)
\end{equation}
thanks to Nualart \cite{NUALART06}, Proposition 5.2.3, where $\Phi_m$ is the map defined on $C^1(\mathbb R)$ by
\begin{displaymath}
\Phi_m(\varphi)(x) :=
\frac{1}{f(x)}\sum_{j = 1}^{m}\left[
\frac{1}{NT}\sum_{i = 1}^{N}\left[I_{i,j} -
\mathfrak a\int_{0}^{T}\int_{0}^{t}
\varphi_j'(X_{t}^{i})\exp\left(\int_{s}^{t}\varphi'(X_{u}^{i})du\right)|t - s|^{2H - 2}dsdt\right]\right]\varphi_j(x)
\end{displaymath}
for every $\varphi\in C^1(\mathbb R)$ and $x\in\mathbb R$, $\mathfrak a :=\sigma^2H(2H - 1)$ and
\begin{displaymath}
I_{i,j} :=\int_{0}^{T}\varphi_j(X_{s}^{i})dX_{s}^{i}
\end{displaymath}
for any $i\in\{1,\dots,N\}$ and $j\in\{1,\dots,m\}$. In the definition of $I_{i,j}$, the stochastic integral is taken pathwise (in the sense of Young).
\\
\\
So, Section \ref{risk_bound_oracle_section} deals with an $\mathbb L^2$-error bound on the auxiliary (pseudo-)estimator $\widehat b_m$, and then Section \ref{risk_bound_section} with the existence, the uniqueness and an $\mathbb L^2$-error bound on the nonparametric estimator $\widetilde b_m$ of $b_0$ defined as the fixed point of $\Phi_m$ in
\begin{displaymath}
\mathcal S_{m,\mathfrak c} :=\{\varphi\in\mathcal S_m :
\|\varphi'\|_{\infty}\leqslant\mathfrak c\}
\quad {\rm with}\quad\mathfrak c >\|b_0'\|_{\infty}
\quad {\rm and}\quad
\mathcal S_m = {\rm span}\{\varphi_1,\dots,\varphi_m\}.
\end{displaymath}
For instance, if $X$ is the stationary solution of Equation (\ref{main_equation}), which exists and is unique under an appropriate dissipativity condition on $b_0$ (see (\ref{dissipativity_condition})), and if $(\varphi_1,\dots,\varphi_m)$ is the $I$-supported trigonometric basis with $I = [\ell,\texttt r]$ ($\ell,\texttt r\in\mathbb R$ such that $\ell <\texttt r$), then the main result of our paper (see Theorem \ref{risk_bound_fixed_point_estimator}) says that
\begin{equation}\label{risk_bound_trigonometric}
\mathbb E\left(\int_{-\infty}^{\infty}(\widetilde b_m(x)\mathbf 1_{\Delta_m} -
b_0(x))^2f_0(x)^2dx\right)
\lesssim
m^{-2} + m^3(N^{-\frac{1}{2}}T^{-1} + T^{2H - 1}),
\end{equation}
where $f_0$ is the (stationary) density of $X_0$ and $\Delta_m$ is an event defined later. By Inequality (\ref{risk_bound_trigonometric}), there is a double compromise to manage with. First, $T$ needs to be of order $N^{-1/(4H)}$, and then $m$ needs to be of order $N^{(2H - 1)/(20H)}$ to get a rate of order $N^{-(2H - 1)/(10H)}$.
\\
\\
Some preliminary results on Skorokhod's integral are given in Section \ref{Skorokhod_integral_section}.
%


%
\section{Basics on the Skorokhod integral}\label{Skorokhod_integral_section}
Let $\langle .,.\rangle_{\mathcal H}$ be the inner product defined by
\begin{displaymath}
\langle h,\eta\rangle_{\mathcal H} :=
\alpha_H
\int_{0}^{T}\int_{0}^{T}
h(s)\eta(t)|t - s|^{2H - 2}dsdt
\quad {\rm with}\quad
\alpha_H = H(2H - 1),
\end{displaymath}
and consider the reproducing kernel Hilbert space $\mathcal H =
\{h :\|h\|_{\mathcal H} <\infty\}$ of $B^T = (B_t)_{t\in [0,T]}$, where $\|.\|_{\mathcal H}$ is the norm associated to $\langle .,.\rangle_{\mathcal H}$. Consider also the isonormal Gaussian process $(\mathbf B(h))_{h\in\mathcal H}$ defined by
\begin{displaymath}
\mathbf B(h) :=
\int_{0}^{T}h(s)dB_s
\textrm{ $;$ }h\in\mathcal H,
\end{displaymath}
which is the Wiener integral of $h$ with respect to $B$ on $[0,T]$.
%


%
\begin{definition}\label{Malliavin_derivative}
The Malliavin derivative of a smooth functional
\begin{displaymath}
F =\varphi(
\mathbf B(h_1),\dots,
\mathbf B(h_n))
\end{displaymath}
where $n\in\mathbb N^*$, $\varphi\in C_{\rm p}^{\infty}(\mathbb R^n;\mathbb R)$ (the space of all the smooth functions $\varphi :\mathbb R^n\rightarrow\mathbb R$ such that $\varphi$ and all its partial derivatives have polynomial growth) and $h_1,\dots,h_n\in\mathcal H$, is the $\mathcal H$-valued random variable
\begin{displaymath}
\mathbf DF :=
\sum_{k = 1}^{n}
\partial_k\varphi(
\mathbf B(h_1),\dots,
\mathbf B(h_n))h_k.
\end{displaymath}
\end{definition}
%


%
\begin{proposition}\label{Malliavin_derivative_domain}
The map $\mathbf D$ is closable from $\mathbb L^2(\Omega;\mathbb R)$ into $\mathbb L^2(\Omega;\mathcal H)$. Its domain in $\mathbb L^2(\Omega;\mathbb R)$, denoted by $\mathbb D^{1,2}$, is the closure of the smooth functionals space for the norm $\|.\|_{1,2}$ defined by
\begin{displaymath}
\|F\|_{1,2}^{2} :=
\mathbb E(F^2) +
\mathbb E(\|\mathbf DF\|_{\mathcal H}^{2}).
\end{displaymath}
The Malliavin derivative of $F\in\mathbb D^{1,2}$ at time $s\in [0,T]$ is denoted by $\mathbf D_sF$.
\end{proposition}
\noindent
See Nualart \cite{NUALART06}, Proposition 1.2.1 for a proof.
%


%
\begin{definition}\label{divergence_operator}
The adjoint $\delta$ of the Malliavin derivative $\mathbf D$ is the divergence operator. The domain of $\delta$ is denoted by $\normalfont{\textrm{dom}}(\delta)$, and $Y\in\normalfont{\textrm{dom}}(\delta)$ if and only if there exists a deterministic constant $\mathfrak c_Y > 0$ such that for every $F\in\mathbb D^{1,2}$,
\begin{displaymath}
|\mathbb E(\langle\mathbf DF,Y\rangle_{\mathcal H})|
\leqslant
\mathfrak c_Y\mathbb E(F^2)^{\frac{1}{2}}.
\end{displaymath}
\end{definition}
\noindent
For any process $Y = (Y_s)_{s\in [0,T]}$ and every $t\in (0,T]$, if $Y\mathbf 1_{[0,t]}\in\textrm{dom}(\delta)$, then its Skorokhod integral with respect to $B^T$ is defined on $[0,t]$ by
\begin{displaymath}
\int_{0}^{t}Y_s\delta B_s :=
\delta(Y\mathbf 1_{[0,t]}),
\end{displaymath}
and its Skorokhod integral with respect to $X$ is defined by
\begin{displaymath}
\int_{0}^{t}
Y_s\delta X_s :=
\int_{0}^{t}Y_sb_0(X_s)ds +
\sigma\int_{0}^{t}Y_s\delta B_s.
\end{displaymath}
Note that since $\delta$ is the adjoint of the Malliavin derivative $\mathbf D$, the Skorokhod integral of $Y$ with respect to $B^T$ on $[0,t]$ is a centered random variable. Indeed,
\begin{equation}\label{zero_mean_Skorokhod_integral}
\mathbb E\left(\int_{0}^{t}Y_s\delta B_s\right) =
\mathbb E(1\cdot\delta(Y\mathbf 1_{[0,t]})) =
\mathbb E(\langle\mathbf D(1),Y\mathbf 1_{[0,t]}\rangle_{\mathcal H}) = 0.
\end{equation}
Let $\mathcal S$ be the space of the smooth functionals presented in Definition \ref{Malliavin_derivative} and consider $\mathbb D^{1,2}(\mathcal H)$, the closure of
\begin{displaymath}
\mathcal S_{\mathcal H} :=
\left\{
\sum_{j = 1}^{n}F_jh_j
\textrm{ $;$ }
h_1,\dots,h_n\in\mathcal H
\textrm{, }
F_1,\dots,F_n\in\mathcal S
\right\}
\end{displaymath}
for the norm $\|.\|_{1,2,\mathcal H}$ defined by
\begin{displaymath}
\|Y\|_{1,2,\mathcal H}^{2} :=\mathbb E(\|Y\|_{\mathcal H}^{2}) +
\mathbb E(\|\mathbf DY\|_{\mathcal H\otimes\mathcal H}^{2}).
\end{displaymath}
Consider also the norm $\|.\|_{\mathfrak H}$ defined by
\begin{displaymath}
\|h\|_{\mathfrak H} :=\left(
\alpha_H\int_{0}^{T}\int_{0}^{T}
|h(s)|\cdot |h(t)|\cdot
|t - s|^{2H - 2}dsdt\right)^{\frac{1}{2}},
\end{displaymath}
the Banach space $\mathfrak H :=\{h :\|h\|_{\mathfrak H} <\infty\}$ and
\begin{displaymath}
\mathbb D^{1,2}(\mathfrak H) :=
\{Y\in\mathbb D^{1,2}(\mathcal H) :
\mathbb E(\|Y\|_{\mathfrak H}^{2}) +
\mathbb E(\|\mathbf DY\|_{\mathfrak H\otimes\mathfrak H}^{2}) <\infty\}.
\end{displaymath}
By Nualart \cite{NUALART06}, Proposition 1.3.1,
\begin{displaymath}
\mathbb D^{1,2}(\mathfrak H)\subset
\mathbb D^{1,2}(\mathcal H)\subset {\rm dom}(\delta).
\end{displaymath}
The two following propositions are crucial in Sections \ref{risk_bound_oracle_section} and \ref{risk_bound_section}.
%


%
\begin{proposition}\label{Skorokhod_pathwise_relationship}
For every $\varphi\in C^1(\mathbb R)$ of bounded derivative, $(\varphi(X_t))_{t\in [0,T]}$ belongs to $\mathbb D^{1,2}(\mathfrak H)$ and
\begin{displaymath}
\int_{0}^{T}\varphi(X_s)\delta X_s =
\int_{0}^{T}\varphi(X_s)dX_s
-\alpha_H\sigma^2
\int_{0}^{T}\int_{0}^{t}\varphi'(X_t)
\exp\left(\int_{s}^{t}b_0'(X_u)du\right)
|t - s|^{2H - 2}dsdt.
\end{displaymath}
\end{proposition}
%


%
\begin{proof}
For any $s,t\in [0,T]$ such that $s < t$, since
\begin{displaymath}
\mathbf D_sB_t =
\mathbf D_s(\mathbf B(\mathbf 1_{[0,t]})) =
\mathbf 1_{[0,t]}(s),
\end{displaymath}
since $X_0$ is a $\sigma((B_u)_{u\in\mathbb R_-})$-measurable random variable, and by Equation (\ref{main_equation}) and Nualart \cite{NUALART06}, Proposition 1.2.3,
\begin{displaymath}
\mathbf D_sX_t =
\int_{0}^{t}b_0'(X_u)\mathbf D_sX_udu +\sigma\mathbf 1_{[0,t]}(s).
\end{displaymath}
Therefore,
\begin{displaymath}
\mathbf D_sX_t =
\sigma\mathbf 1_{[0,t]}(s)\exp\left(\int_{s}^{t}b_0'(X_u)du\right)
\end{displaymath}
and, by Nualart \cite{NUALART06}, Propositions 1.2.3 and 5.2.3,
\begin{eqnarray*}
 \int_{0}^{T}\varphi(X_s)\delta X_s & = &
 \int_{0}^{T}\varphi(X_s)dX_s -
 \alpha_H\sigma
 \int_{0}^{T}\int_{0}^{T}\mathbf D_s(\varphi(X_t))|t - s|^{2H - 2}dsdt\\
 & = &
 \int_{0}^{T}\varphi(X_s)dX_s\\
 & &
 \hspace{1cm} -
 \alpha_H\sigma^2
 \int_{0}^{T}\int_{0}^{t}\varphi'(X_t)
 \exp\left(\int_{s}^{t}b_0'(X_u)du\right)
 |t - s|^{2H - 2}dsdt.
\end{eqnarray*}
\end{proof}
%


%
\begin{proposition}\label{bound_variance_Skorokhod}
Consider
\begin{displaymath}
M :=\sup_{x\in\mathbb R}b_0'(x).
\end{displaymath}
There exists a constant $\mathfrak c_{\ref{bound_variance_Skorokhod}} > 0$, only depending on $H$ and $\sigma$, such that for every $\varphi\in C^1(\mathbb R)$ of bounded derivative,
\begin{eqnarray*}
 \mathbb E\left[\left(\int_{0}^{T}\varphi(X_s)\delta B_s\right)^2\right]
 & \leqslant &
 \mathfrak c_{\ref{bound_variance_Skorokhod}}\overline{\mathfrak m}_TT^{2H - 1}
 \left(\int_{0}^{T}\mathbb E(\varphi(X_s)^2)ds +
 \int_{0}^{T}\mathbb E(\varphi'(X_s)^2)ds\right)
\end{eqnarray*}
with $\overline{\mathfrak m}_T = 1\vee\mathfrak m_T$ and
\begin{displaymath}
\mathfrak m_T =
\left(-\frac{H}{M}\right)^{2H}\mathbf 1_{M < 0} +
T^{2H}\mathbf 1_{M = 0} +
\left(\frac{H}{M}\right)^{2H}e^{2MT}\mathbf 1_{M > 0}.
\end{displaymath}
\end{proposition}
\noindent
See Hu et al. \cite{HNZ19}, Proposition 4.4.(2) and Comte and Marie \cite{CM21}, Theorem 2.9 for a proof.
\\
\\
For details on the Malliavin calculus, the reader may refer to Decreusefond \cite{DECREUSEFOND22} and Nualart \cite{NUALART06}.
%


%
\section{An $\mathbb L^2$-error bound on the auxiliary estimator $\widehat b_m$}\label{risk_bound_oracle_section}
First, the probability distribution of $X_t$, $t\in (0,T]$, and the function $b_0$ need to fulfill the following assumption.
%


%
\begin{assumption}\label{conditions_probability_distribution_X_t}
For every $t\in (0,T]$, the probability distribution of $X_t$ has a density $f_t$ with respect to Lebesgue's measure such that:
\begin{enumerate}
 \item The function $t\mapsto f_t(x)$ belongs to $\mathbb L^1([0,T])$ for every $x\in\mathbb R$.
 \item The function $|b_0|^{\alpha}$ belongs to $\mathbb L^2(\mathbb R,f(x)dx)$ for every $\alpha\in\mathbb R_+$, and $\mathcal S_m$ is a subset of $\mathbb L^2(\mathbb R,f(x)dx)$, where $f$ is the density function defined by
 \begin{displaymath}
 f(x) :=\frac{1}{T}\int_{0}^{T}f_s(x)ds
 \textrm{ $;$ }\forall x\in\mathbb R.
 \end{displaymath}
\end{enumerate}
\end{assumption}
%


%
\begin{example}\label{examples_conditions_probability_distribution_X_t}
In the two following situations, Assumption \ref{conditions_probability_distribution_X_t} is fulfilled:
\begin{enumerate}
 \item Assume that $X_0(.) = x_0$ with $x_0\in\mathbb R$. For every $t\in (0,T]$, the probability distribution of $X_t$ has a density $f_t$ with respect to Lebesgue's measure such that, for every $x\in\mathbb R$,
 \begin{equation}\label{Gaussian_bound_density}
 f_t(x)\leqslant
 \mathfrak c_{H,T}t^{-H}
 \exp\left(-\mathfrak m_{H,T}\frac{(x - x_0)^2}{t^{2H}}\right)
 \end{equation}
 where $\mathfrak c_{H,T}$ and $\mathfrak m_{H,T}$ are positive constants depending on $T$ but not on $t$ and $x$ (see Li et al. \cite{LPS23}, Theorem 1.3). So, $t\mapsto f_t(x)$ belongs to $\mathbb L^1([0,T])$. Still by Inequality (\ref{Gaussian_bound_density}),
 \begin{displaymath}
 |b_0|^{\alpha}\in\mathbb L^2(\mathbb R,f(x)dx)
 \textrm{ $;$ }
 \forall\alpha\in\mathbb R_+
 \end{displaymath}
 because $b_0'$ is bounded, and $\mathcal S_m\subset\mathbb L^2(\mathbb R,f(x)dx)$ because the $\varphi_j$'s belong to $\mathbb L^2(\mathbb R,dx)$.
 \item Assume that $b_0$ satisfies the dissipativity condition
 \begin{equation}\label{dissipativity_condition}
 \exists\mathfrak m > 0 :\forall x\in\mathbb R\textrm{, }b_0'(x)\leqslant -\mathfrak m.
 \end{equation}
 Then, Equation (\ref{main_equation}) has a unique stationary solution $X$ with a $\sigma((B_t)_{t\in\mathbb R_-})$-measurable initial condition $X_0$, and the common probability distribution of the $X_t$'s has a density $f_0$ with respect to Lebesgue's measure such that, for every $x\in\mathbb R$,
 \begin{displaymath}
 f_0(x)\leqslant\mathfrak c_He^{-\mathfrak m_Hx^2},
 \end{displaymath}
 where $\mathfrak c_H$ and $\mathfrak m_H$ are positive constants not depending on $x$ (see Li et al. \cite{LPS23}, Theorem 1.1). So, $b_0$ and $f = f_0$ fulfill Assumption \ref{conditions_probability_distribution_X_t}.(2). Of course, Assumption \ref{conditions_probability_distribution_X_t}.(1) is fulfilled with $f_t = f_0$ for every $t\in (0,T]$.
\end{enumerate}
\end{example}
\noindent
{\bf Notations:}
\begin{itemize}
 \item For every measurable function $\varphi :\mathbb R\rightarrow\mathbb R$,
 \begin{displaymath}
 \|\varphi\|^2 :=\int_I\varphi(x)^2dx,
 \quad
 \|\varphi\|_{f}^{2} :=\int_I\varphi(x)^2f(x)dx
 \quad {\rm and}\quad
 \|\varphi\|_{f^2}^{2} :=\int_I\varphi(x)^2f(x)^2dx.
 \end{displaymath}
 \item The orthonormal projection of $b_0f$ on $\mathcal S_m$ is denoted by $(bf)_m$ in the sequel:
 \begin{displaymath}
 (bf)_m :=\sum_{j = 1}^{m}\langle b_0f,\varphi_j\rangle\varphi_j.
 \end{displaymath}
 Moreover, $b_m := (bf)_m/f$.
\end{itemize}
The following proposition provides a suitable $\mathbb L^2$-error bound on $\widehat b_m$.
%


%
\begin{proposition}\label{risk_bound_oracle}
Under Assumption \ref{conditions_probability_distribution_X_t},
\begin{displaymath}
\mathbb E(\|\widehat b_m - b_0\|_{f^2}^{2})\leqslant
\|b_m - b_0\|_{f^2}^{2} +
\frac{2}{N}\left(\|b_0\|_{f}^{2}\mathfrak L(m) +
\mathfrak c_{\ref{bound_variance_Skorokhod}}\sigma^2
\frac{\overline{\mathfrak m}_T}{T^{2 - 2H}}(\mathfrak L(m) +\mathfrak R(m))\right),
\end{displaymath}
where
\begin{displaymath}
\mathfrak L(m) :=
\sum_{j = 1}^{m}\|\varphi_j\|_{\infty}^{2}
\quad\textrm{ and }\quad
\mathfrak R(m) :=
\sum_{j = 1}^{m}\|\varphi_j'\|_{\infty}^{2}.
\end{displaymath}
\end{proposition}
%


%
\begin{proof}
First of all,
\begin{eqnarray}
 \mathbb E(\|\widehat b_m - b_0\|_{f^2}^{2}) & = &
 \mathbb E(\|\widehat{bf}_m - b_0f\|^2)
 \nonumber\\
 \label{risk_bound_oracle_1}
 & = &
 \|(bf)_m - b_0f\|^2 +
 \mathbb E(\|\widehat{bf}_m - (bf)_m\|^2)
\end{eqnarray}
because $\widehat{bf}_m(\omega)$ ($\omega\in\Omega$) and $(bf)_m$ belong to $\mathcal S_m$, and $(bf)_m$ is the orthogonal projection of $b_0f$ on $\mathcal S_m$. Since $X^1,\dots,X^N$ are independent copies of $X$, and since the Skorokhod integral with respect to $B^T$ is centered, for every $j\in\{1,\dots,m\}$,
\begin{eqnarray*}
 \mathbb E\left(\frac{1}{NT}\sum_{i = 1}^{N}\int_{0}^{T}\varphi_j(X_{s}^{i})\delta X_{s}^{i}\right)
 & = &
 \frac{1}{T}\mathbb E\left(\int_{0}^{T}\varphi_j(X_s)b_0(X_s)ds\right)\\
 & = &
 \frac{1}{T}\int_{-\infty}^{\infty}\varphi_j(x)b_0(x)\int_{0}^{T}f_s(x)dsdx =
 \langle b_0f,\varphi_j\rangle.
\end{eqnarray*}
Then, since $(\varphi_1,\dots,\varphi_m)$ is an orthonormal family in $\mathbb L^2(\mathbb R,dx)$,
\begin{eqnarray*}
 \mathbb E(\|\widehat{bf}_m - (bf)_m\|^2) & = &
 \mathbb E\left[
 \sum_{j = 1}^{m}\left(
 \frac{1}{NT}\sum_{i = 1}^{N}\int_{0}^{T}\varphi_j(X_{s}^{i})\delta X_{s}^{i} -
 \langle b_0f,\varphi_j\rangle\right)^2\right]\\
 & = &
 \sum_{j = 1}^{m}
 {\rm var}\left(
 \frac{1}{NT}\sum_{i = 1}^{N}\int_{0}^{T}\varphi_j(X_{s}^{i})\delta X_{s}^{i}\right)\\
 & \leqslant &
 \frac{2}{N}\sum_{j = 1}^{m}
 \left[
 \mathbb E\left[\left(\frac{1}{T}\int_{0}^{T}\varphi_j(X_s)b_0(X_s)ds\right)^2\right] +
 \sigma^2\mathbb E\left[\left(
 \frac{1}{T}\int_{0}^{T}\varphi_j(X_s)\delta B_s\right)^2\right]
 \right].
\end{eqnarray*}
Moreover,
\begin{eqnarray*}
 \sum_{j = 1}^{m}
 \mathbb E\left[\left(\frac{1}{T}\int_{0}^{T}\varphi_j(X_s)b_0(X_s)ds\right)^2\right]
 & \leqslant &
 \frac{1}{T}
 \sum_{j = 1}^{m}\int_{0}^{T}\mathbb E(\varphi_j(X_s)^2b_0(X_s)^2)ds\\
 & = &
 \sum_{j = 1}^{m}
 \int_{-\infty}^{\infty}\varphi_j(x)^2b_0(x)^2f(x)dx
 \leqslant
 \mathfrak L(m)\|b_0\|_{f}^{2}
\end{eqnarray*}
and, by Proposition \ref{bound_variance_Skorokhod},
\begin{eqnarray*}
 \sum_{j = 1}^{m}
 \mathbb E\left[\left(
 \frac{1}{T}\int_{0}^{T}\varphi_j(X_s)\delta B_s\right)^2\right]
 & \leqslant &
 \mathfrak c_{\ref{bound_variance_Skorokhod}}
 \frac{\overline{\mathfrak m}_T}{T^{2 - 2H}}
 \sum_{j = 1}^{m}\int_{-\infty}^{\infty}(\varphi_j(x)^2 +\varphi_j'(x)^2)f(x)dx\\
 & \leqslant &
 \mathfrak c_{\ref{bound_variance_Skorokhod}}
 \frac{\overline{\mathfrak m}_T}{T^{2 - 2H}}(\mathfrak L(m) +\mathfrak R(m)).
\end{eqnarray*}
Therefore, by Equality (\ref{risk_bound_oracle_1}),
\begin{displaymath}
\mathbb E(\|\widehat b_m - b_0\|_{f^2}^{2})\leqslant
\|(bf)_m - b_0f\|^2 +
\frac{2}{N}\left(\|b_0\|_{f}^{2}\mathfrak L(m) +
\mathfrak c_{\ref{bound_variance_Skorokhod}}\sigma^2
\frac{\overline{\mathfrak m}_T}{T^{2 - 2H}}(\mathfrak L(m) +\mathfrak R(m))\right).
\end{displaymath}
\end{proof}
%


%
\section{Existence, uniqueness and $\mathbb L^2$-error bound on the fixed point estimator $\widetilde b_m$}\label{risk_bound_section}
The Skorokhod integral, and then $\widehat b_m$ are uncomputable. However, by Proposition \ref{Skorokhod_pathwise_relationship},
\begin{displaymath}
\int_{0}^{T}\varphi_j(X_{s}^{i})\delta X_{s}^{i} =
\underbrace{
\int_{0}^{T}\varphi_j(X_{s}^{i})dX_{s}^{i}}_{= I_{i,j}} -
\alpha_H\sigma^2\int_{0}^{T}\int_{0}^{t}
\varphi_j'(X_{t}^{i})\exp\left(\int_{s}^{t}b_0'(X_{u}^{i})du\right)|t - s|^{2H - 2}dsdt
\end{displaymath}
for every $i\in\{1,\dots,N\}$ and $j\in\{1,\dots,m\}$. Then, $\widehat b_m =\Phi_m(b_0)$, where $\Phi_m$ has been defined in Section \ref{section_introduction} by
\begin{displaymath}
\Phi_m(\varphi)(x) :=
\frac{1}{f(x)}\sum_{j = 1}^{m}\left[
\frac{1}{NT}\sum_{i = 1}^{N}\left[I_{i,j} -
\mathfrak a\int_{0}^{T}\int_{0}^{t}
\varphi_j'(X_{t}^{i})\exp\left(\int_{s}^{t}\varphi'(X_{u}^{i})du\right)|t - s|^{2H - 2}dsdt\right]\right]\varphi_j(x)
\end{displaymath}
for every $\varphi\in C^1(\mathbb R)$ and $x\in\mathbb R$. Since $\widehat b_m$ is a converging estimator of $b_0$ as established in Section \ref{risk_bound_oracle_section},
\begin{displaymath}
\widehat b_m =
\Phi_m(b_0)\approx
\Phi_m(\widehat b_m),
\end{displaymath}
which legitimates to consider the fixed point $\widetilde b_m$ of $\Phi_m$ as a computable estimator of $b_0$. This section deals with the existence, the uniqueness (thanks to Picard's theorem) and an $\mathbb L^2$-error bound on $\widetilde b_m$.
\\
\\
In the sequel, consider the map $F_m$ from $\mathbb R^m$ into itself such that
\begin{displaymath}
F_m(\theta)_j :=
\frac{\mathfrak a}{NT}\sum_{i = 1}^{N}
\int_{0}^{T}\int_{0}^{t}\varphi_j'(X_{t}^{i})
\exp\left(\sum_{\ell = 1}^{m}\theta_{\ell}\int_{s}^{t}\varphi_{\ell}'(X_{u}^{i})du\right)
|t - s|^{2H - 2}dsdt
\end{displaymath}
for every $\theta\in\mathbb R^m$ and $j\in\{1,\dots,m\}$. Moreover, $I$ is compact and the density function $f$ fulfills the following additional assumption.
%


%
\begin{assumption}\label{lower_bound_f}
There exists $\mathfrak m_f > 0$ such that
\begin{displaymath}
f(x)\geqslant\mathfrak m_f
\textrm{ $;$ }\forall x\in I.
\end{displaymath}
\end{assumption}
%


%
\begin{example}\label{examples_lower_bound_f}
In the two following situations, the density function $f$ fulfills Assumption \ref{lower_bound_f}:
\begin{enumerate}
 \item Assume that $X_0(.) = x_0$ with $x_0\in\mathbb R$. By Li et al. \cite{LPS23}, Theorem 1.3, there exist two constants $\overline{\mathfrak c}_{H,T},\overline{\mathfrak m}_{H,T} > 0$ such that, for every $t\in (0,T]$ and $x\in I$,
 \begin{eqnarray*}
  f_t(x) & \geqslant &
  \overline{\mathfrak c}_{H,T}t^{-H}
  \exp\left(-\overline{\mathfrak m}_{H,T}\frac{(x - x_0)^2}{t^{2H}}\right)\\
  & \geqslant &
  u_I(t) :=\overline{\mathfrak c}_{H,T}t^{-H}
  \exp\left(-\overline{\mathfrak m}_{H,T}\frac{x_{I}^{2}}{t^{2H}}\right) > 0
 \end{eqnarray*}
 with $x_I = |\min(I -\{x_0\})|\vee |\max(I -\{x_0\})|$. Then, $f$ fulfills Assumption \ref{lower_bound_f} with
 \begin{displaymath}
 \mathfrak m_f =\frac{1}{T}\int_{0}^{T}u_I(s)ds > 0.
 \end{displaymath}
 \item Assume that $b_0$ fulfills the dissipativity condition (\ref{dissipativity_condition}). By Li et al. \cite{LPS23}, Theorem 1.1, there exist two constants $\overline{\mathfrak c}_H,\overline{\mathfrak m}_H > 0$ such that, for every $x\in I$,
 \begin{displaymath}
 f_0(x)\geqslant
 \overline{\mathfrak c}_H
 e^{-\overline{\mathfrak m}_Hx^2}
 \geqslant
 \overline{\mathfrak c}_H
 e^{-\overline{\mathfrak m}_Hx_{I}^{2}} > 0
 \end{displaymath}
 with $x_I = |\min(I)|\vee |\max(I)|$. Then, $f = f_0$ fulfills Assumption \ref{lower_bound_f} with $\mathfrak m_f =\overline{\mathfrak c}_He^{-\overline{\mathfrak m}_Hx_{I}^{2}}$.
\end{enumerate}
\end{example}
\noindent
Note that under Assumption \ref{lower_bound_f}, the norms $\|.\|$ and $\|.\|_{f^2}$ are equivalent on $\mathcal S_{m,\mathfrak c}$. Indeed, for every $\varphi\in\mathcal S_{m,\mathfrak c}$,
\begin{displaymath}
\|\varphi\|_{f^2}\leqslant\|f\|_{\infty}^{2}\|\varphi\|
\quad {\rm and}\quad
\|\varphi\|\leqslant
\frac{1}{\mathfrak m_{f}^{2}}\|\varphi\|_{f^2}.
\end{displaymath}
{\bf Notations:}
\begin{itemize}
 \item For any $\varphi\in\mathcal S_m$, its coordinate vector in the basis $(\varphi_1,\dots,\varphi_m)$ is denoted by $\theta^{\varphi} = (\theta_{1}^{\varphi},\dots,\theta_{m}^{\varphi})$.
 \item The usual norm on $\mathcal L(\mathbb R^m)$ is denoted by $\|.\|_{\rm op}$.
\end{itemize}
First, an event on which the map $\Phi_m$ has a unique fixed point in $\mathcal S_{m,\mathfrak c}$ is provided in the following proposition.
%


%
\begin{proposition}\label{existence_uniqueness_fixed_point}
Under Assumptions \ref{conditions_probability_distribution_X_t} and \ref{lower_bound_f}, consider $\mathfrak l\in (0,1)$ and the event $\Delta_m :=\Delta_{m,\mathfrak c}\cap\Delta_{m,\mathfrak l}$, where
\begin{displaymath}
\Delta_{m,\mathfrak c} :=
\left\{\sup_{\varphi\in\mathcal S_{m,\mathfrak c}}\|\Phi_m(\varphi)'\|_{\infty}\leqslant\mathfrak c\right\}
\quad\textrm{and}\quad
\Delta_{m,\mathfrak l} :=\left\{\frac{1}{\mathfrak m_{f}^{2}}
\sup_{\varphi\in\mathcal S_{m,\mathfrak c}}\|DF_m(\theta^{\varphi})\|_{\rm op}^{2}\leqslant\mathfrak l^2\right\}.
\end{displaymath}
On the event $\Delta_m$, $\Phi_m$ is a contraction from $\mathcal S_{m,\mathfrak c}$ into itself, and then it has a unique fixed point $\widetilde b_m$ in $\mathcal S_{m,\mathfrak c}$.
\end{proposition}
%


%
\begin{proof}
First, by the definition of $\Delta_{m,\mathfrak c}$, $\Phi_m(\mathcal S_{m,\mathfrak c})\subset\mathcal S_{m,\mathfrak c}$ on $\Delta_{m,\mathfrak c}$. For any $\varphi,\overline\varphi\in\mathcal S_{m,\mathfrak c}$, since $(\varphi_1,\dots,\varphi_m)$ is an orthonormal family of $\mathbb L^2(\mathbb R,dx)$,
\begin{eqnarray*}
 \|\Phi_m(\varphi) -\Phi_m(\overline\varphi)\|_{f^2}^{2} & = &
 \|F_m(\theta^{\varphi}) - F_m(\theta^{\overline\varphi})\|_{m}^{2}\\
 & \leqslant &
 \left(\sup_{\psi\in\mathcal S_{m,\mathfrak c}}\|DF_m(\theta^{\psi})\|_{\rm op}\right)^2
 \|\theta^{\varphi} -\theta^{\overline\varphi}\|_{m}^{2}
 \leqslant
 \left(\frac{1}{\mathfrak m_{f}^{2}}
 \sup_{\psi\in\mathcal S_{m,\mathfrak c}}\|DF_m(\theta^{\psi})\|_{\rm op}^{2}\right)
 \|\varphi -\overline\varphi\|_{f^2}^{2}.
\end{eqnarray*}
This leads to
\begin{displaymath}
\|\Phi_m(\varphi) -\Phi_m(\overline\varphi)\|_{f^2}\leqslant
\mathfrak l\|\varphi -\overline\varphi\|_{f^2}
\quad {\rm on}\quad\Delta_{m,\mathfrak l}.
\end{displaymath}
Therefore, $\Phi_m$ is a contraction from $\mathcal S_{m,\mathfrak c}$ into itself on the event $\Delta_m$.
\end{proof}
\noindent
Now, assume that $b_0$ fulfills the dissipativity condition (\ref{dissipativity_condition}). Then, as already mentioned in Section \ref{risk_bound_oracle_section}, Equation (\ref{main_equation}) has a unique stationary solution $X$, and $f = f_0$ is the density of the common probability distribution of the $X_t$'s. Assumptions \ref{conditions_probability_distribution_X_t} and \ref{lower_bound_f} are fulfilled as established in Example \ref{examples_conditions_probability_distribution_X_t}.(2) and Example \ref{examples_lower_bound_f}.(2) respectively. Moreover, consider
\begin{displaymath}
\mathfrak I(m) :=
\sum_{j = 1}^{m}\|\overline\varphi_j\|_{\infty}^2
\end{displaymath}
where, for every $j\in\{1,\dots,m\}$, $\overline\varphi_j$ is a primitive function of $\varphi_j$. In the sequel, $X^1,\dots,X^N$ are independent copies of the stationary solution $X$ of Equation (\ref{main_equation}), and the $\varphi_j$'s fulfill the following additional assumption.
%


%
\begin{assumption}\label{assumption_LRI}
There exists a constant $\mathfrak c_{\varphi} > 0$, not depending on $m$, such that
\begin{displaymath}
\mathfrak I(m)\vee\mathfrak L(m)\leqslant\mathfrak c_{\varphi}\mathfrak R(m).
\end{displaymath}
\end{assumption}
%


%
\begin{example}\label{example_assumption_LRI}
Assume that $I = [\ell,{\tt r}]$ with $\ell,{\tt r}\in\mathbb R$ satisfying $\ell < {\tt r}$, and that $(\varphi_1,\dots,\varphi_m)$ is the trigonometric basis:
\begin{eqnarray*}
 \varphi_1(x) & := & \sqrt{\frac{1}{{\tt r} -\ell}}
 \mathbf 1_I(x),\\
 \varphi_{2j + 1}(x) & := & \sqrt{\frac{2}{{\tt r} -\ell}}
 \sin\left(2\pi j\frac{x -\ell}{{\tt r} -\ell}\right)
 \mathbf 1_I(x)\quad {\rm and}\\
 \varphi_{2j}(x) & := & \sqrt{\frac{2}{{\tt r} -\ell}}\cos\left(2\pi j\frac{x -\ell}{{\tt r} -\ell}\right)
 \mathbf 1_I(x)
\end{eqnarray*}
for every $x\in I$ and $j\in\mathbb N^*$ satisfying $2j + 1\leqslant m$. Let us show that the trigonometric basis fulfills Assumption \ref{assumption_LRI}. On the one hand,
\begin{displaymath}
\mathfrak R(m) =
\underbrace{\frac{4\pi^2}{({\tt r} -\ell)^3}}_{=:\mathfrak c_1}m^3\geqslant
\mathfrak c_1m =\mathfrak c_1\mathfrak L(m)
\end{displaymath}
as established in Comte and Marie \cite{CM21}, Section 3.2.1. On the other hand, assume that for every $j\in\{1,\dots,m\}$, the primitive function $\overline\varphi_j$ of $\varphi_j$ is chosen such that $\overline\varphi_j(x_j) = 0$ with $x_j\in I$. Then,
\begin{eqnarray*}
 |\overline\varphi_j(.)| & = &
 \left|\int_{x_j}^{.}\varphi_j(y)dy\right|\\
 & \leqslant &
 \lambda_I\|\varphi_j\|_{\infty}
 \quad {\rm with}\quad
 \lambda_I =\max(I) -\min(I),
\end{eqnarray*}
leading to $\mathfrak I(m)\leqslant\lambda_{I}^{2}\mathfrak L(m)$.
\end{example}
\noindent
In order to provide an $\mathbb L^2$-error bound on the truncated estimator
\begin{displaymath}
\widetilde b_{m}^{\mathfrak c,\mathfrak l} :=\widetilde b_m\mathbf 1_{\Delta_m}
\end{displaymath}
in Theorem \ref{risk_bound_fixed_point_estimator}, a suitable control of $\mathbb P(\Delta_{m}^{c})$ needs to be established first.
%


%
\begin{lemma}\label{bound_complement_Omega}
If $b_0$ fulfills the dissipativity condition (\ref{dissipativity_condition}), and if $X^1,\dots,X^N$ are independent copies of the stationary solution $X$ of Equation (\ref{main_equation}), then there exists a constant $\mathfrak c_{\ref{bound_complement_Omega},1} > 0$, not depending on $m$, $N$ and $T$, such that
\begin{displaymath}
\mathbb P(\Delta_{m}^{c})
\leqslant
\mathfrak c_{\ref{bound_complement_Omega},1}(
(\mathfrak L(m)^{\frac{1}{2}} +\mathfrak R(m)^{\frac{1}{2}})
(N^{-\frac{1}{2}}\mathfrak I(m)^{\frac{1}{2}}T^{-1} +
\mathfrak R(m)^{\frac{1}{2}}e^{\mathfrak cT}T^{2H - 1}) +
\mathfrak R(m)e^{\mathfrak cT}T^{2H}).
\end{displaymath}
Under Assumption \ref{assumption_LRI}, if $T\in (0,1]$, then there exists a constant $\mathfrak c_{\ref{bound_complement_Omega},2} > 0$, not depending on $m$, $N$ and $T$, such that
\begin{equation}\label{bound_complement_Omega_1}
\mathbb P(\Delta_{m}^{c})
\leqslant\mathfrak c_{\ref{bound_complement_Omega},2}\mathfrak R(m)(
N^{-\frac{1}{2}}T^{-1} + T^{2H - 1}).
\end{equation}
\end{lemma}
\noindent
Inequality (\ref{bound_complement_Omega_1}) suggests that, at the end, there will be a double compromise to manage with. First, the usual bias-variance tradeoff between the number $N$ of observations and the dimension $m$ of the projection space, but also a compromise between $N$ and $T$. Indeed, the time horizon $T$ should be chosen near to $0$ in order to minimize $T^{2H - 1}$, but not that much because $N^{-1/2}T^{-1}$ and the variance term in the $\mathbb L^2$-error bound on the auxiliary estimator $\widehat b_m$ (see Proposition \ref{risk_bound_oracle}) explode when $T$ goes to $0$. See the final remarks for details. Finally, note that if $X^1,\dots,X^N$ have been observed on $[0,T_{\max}]$ for a fixed $T_{\max} > 0$, then our estimator is calculable for $T = T(N)$ with $T(N)\downarrow 0$ when $N\rightarrow\infty$.
%


%
\begin{proof}
The proof of Lemma \ref{bound_complement_Omega} is dissected in three steps. Step 1 deals with a bound on $\mathbb P(\Delta_{m,\mathfrak c}^{c})$, Step 2 with a bound on $\mathbb P(\Delta_{m,\mathfrak l}^{c})$, and the conclusion comes in Step 3.
\\
\\
{\bf Step 1.} Consider $\varphi\in\mathcal S_{m,\mathfrak c}$. Since $f = f_0\in C^1(\mathbb R)$ and $f'$ is bounded by Li et al. \cite{LPS23}, Theorem 1.1, for every $x\in\mathbb R$,
\begin{eqnarray*}
 |\Phi_m(\varphi)'(x)|
 & \leqslant &
 \left|\frac{f'(x)}{f(x)^2}\right|\sum_{j = 1}^{m}|\varphi_j(x)|U_j(\varphi) +
 \left|\frac{1}{f(x)}\right|\sum_{j = 1}^{m}|\varphi_j'(x)|U_j(\varphi)\\
 & \leqslant &
 \underbrace{\left(\frac{\|f'\|_{\infty}}{\mathfrak m_{f}^{2}}\vee
 \frac{1}{\mathfrak m_f}\right)}_{=:\mathfrak m_{f,f'}}
 \sum_{j = 1}^{m}(|\varphi_j(x)| + |\varphi_j'(x)|)U_j(\varphi)
\end{eqnarray*}
where, for every $j\in\{1,\dots,m\}$,
\begin{eqnarray*}
 U_j(\varphi) & := &
 \frac{1}{NT}\left|\sum_{i = 1}^{N}
 \left[I_{i,j} -\mathfrak a\int_{0}^{T}\int_{0}^{t}\varphi_j'(X_{t}^{i})
 \exp\left(\int_{s}^{t}\varphi'(X_{u}^{i})du\right)|t - s|^{2H - 2}dsdt\right]\right|\\
 & \leqslant &
 \left|\frac{1}{NT}\sum_{i = 1}^{N}\int_{0}^{T}\varphi_j(X_{s}^{i})dX_{s}^{i}\right| +
 \frac{\mathfrak a}{NT}\|\varphi_j'\|_{\infty}e^{\mathfrak cT}
 \left(N\int_{0}^{T}\int_{0}^{t}|t - s|^{2H - 2}dsdt\right)\\
 & = &
 \left|\frac{1}{NT}\sum_{i = 1}^{N}
 (\overline\varphi_j(X_{T}^{i}) -\overline\varphi_j(X_{0}^{i}))\right| +
 \frac{\mathfrak a}{2H(2H - 1)}\|\varphi_j'\|_{\infty}e^{\mathfrak cT}T^{2H - 1}
\end{eqnarray*}
by the change of variable formula for Young's integral. So,
\begin{eqnarray*}
 \|\Phi_m(\varphi)'\|_{\infty} & \leqslant &
 \mathfrak m_{f,f'}\sum_{j = 1}^{m}(\|\varphi_j\|_{\infty} +\|\varphi_j'\|_{\infty})\\
 & &
 \hspace{2cm}\times\left(
 \left|\frac{1}{NT}\sum_{i = 1}^{N}
 (\overline\varphi_j(X_{T}^{i}) -\overline\varphi_j(X_{0}^{i}))\right| +
 \frac{\mathfrak a}{2H(2H - 1)}\|\varphi_j'\|_{\infty}e^{\mathfrak cT}T^{2H - 1}
 \right) =: V_m.
\end{eqnarray*}
Moreover, since $X$ is a stationary process, $f = f_0 = f_T$, and then
\begin{eqnarray*}
 \mathbb E\left(\left|\frac{1}{NT}\sum_{i = 1}^{N}
 (\overline\varphi_j(X_{T}^{i}) -\overline\varphi_j(X_{0}^{i}))\right|\right)
 & \leqslant &
 {\rm var}\left(\frac{1}{NT}\sum_{i = 1}^{N}
 (\overline\varphi_j(X_{T}^{i}) -\overline\varphi_j(X_{0}^{i}))\right)^{\frac{1}{2}}\\
 & \leqslant &
 \frac{1}{N^{1/2}T}\mathbb E((\overline\varphi_j(X_T) -\overline\varphi_j(X_0))^2)^{\frac{1}{2}}
 \leqslant\frac{2\|\overline\varphi_j\|_{\infty}}{N^{1/2}T}
\end{eqnarray*}
for every $j\in\{1,\dots,m\}$. This leads to
\begin{eqnarray*}
 \mathbb E\left(
 \sum_{j = 1}^{m}(\|\varphi_j\|_{\infty} +\|\varphi_j'\|_{\infty})
 \left|\frac{1}{NT}\sum_{i = 1}^{N}
 (\overline\varphi_j(X_{T}^{i}) -\overline\varphi_j(X_{0}^{i}))\right|\right)
 & \leqslant &
 \frac{2}{N^{1/2}T}
 \sum_{j = 1}^{m}(\|\varphi_j\|_{\infty} +\|\varphi_j'\|_{\infty})\|\overline\varphi_j\|_{\infty}\\
 & \leqslant &
 \frac{2}{N^{1/2}T}(\mathfrak L(m)^{\frac{1}{2}} +
 \mathfrak R(m)^{\frac{1}{2}})\mathfrak I(m)^{\frac{1}{2}}.
\end{eqnarray*}
Therefore, by Markov's inequality,
\begin{eqnarray*}
 \mathbb P(\Delta_{m,\mathfrak c}^{c})
 & \leqslant &
 \frac{\mathbb E(|V_m|)}{\mathfrak c}\\
 & \leqslant &
 \frac{\mathfrak m_{f,f'}}{\mathfrak c}
 (\mathfrak L(m)^{\frac{1}{2}} +\mathfrak R(m)^{\frac{1}{2}})
 \left(
 2N^{-\frac{1}{2}}\mathfrak I(m)^{\frac{1}{2}}T^{-1} +
 \frac{\mathfrak a}{2H(2H - 1)}
 \mathfrak R(m)^{\frac{1}{2}}e^{\mathfrak cT}T^{2H - 1}
 \right).
\end{eqnarray*}
{\bf Step 2.} For every $\theta\in\mathbb R^m$,
\begin{eqnarray*}
 \|DF_m(\theta)\|_{\rm op}^{2} & = &
 \left(\sup_{z:\|z\|_m = 1}
 \|DF_m(\theta).z\|_m\right)^2\\
 & = &
 \sup_{z:\|z\|_m = 1}
 \sum_{j = 1}^{m}\left(\sum_{\ell = 1}^{m}[\nabla F_m(\theta)_j]_{\ell}z_{\ell}\right)^2
 \leqslant
 \sum_{j,\ell = 1}^{m}[\nabla F_m(\theta)_j]_{\ell}^{2}.
\end{eqnarray*}
Then, for every $\varphi\in\mathcal S_{m,\mathfrak c}$,
\begin{eqnarray*}
 \|DF_m(\theta^{\varphi})\|_{\rm op}^{2}
 & \leqslant &
 \sum_{j,\ell = 1}^{m}\left[
 \frac{\mathfrak a}{NT}\sum_{i = 1}^{N}\int_{0}^{T}\int_{0}^{t}
 \varphi_j'(X_{t}^{i})\exp\left(
 \int_{s}^{t}\varphi'(X_{u}^{i})du\right)
 \left(\int_{s}^{t}\varphi_{\ell}'(X_{u}^{i})du\right)|t - s|^{2H - 2}dsdt
 \right]^2\\
 & \leqslant &
 \frac{\mathfrak a^2}{N^2T^2}\mathfrak R(m)^2e^{2\mathfrak cT}
 \left[N\int_{0}^{T}\int_{0}^{t}|t - s|^{2H - 1}dsdt\right]^2
 =\left(\frac{\mathfrak a}{2H(2H + 1)}\right)^2
 \mathfrak R(m)^2e^{2\mathfrak cT}T^{4H}.
\end{eqnarray*}
So, by Markov's inequality,
\begin{displaymath}
\mathbb P(\Delta_{m,\mathfrak l}^{c})
\leqslant
\frac{1}{\mathfrak m_f\mathfrak l}
\mathbb E\left[\left(\sup_{\varphi\in\mathcal S_{m,\mathfrak c}}
\|DF_m(\theta^{\varphi})\|_{\rm op}^{2}\right)^{\frac{1}{2}}\right]
\leqslant
\frac{1}{\mathfrak m_f\mathfrak l}
\cdot\frac{\mathfrak a}{2H(2H + 1)}
\mathfrak R(m)e^{\mathfrak cT}T^{2H}.
\end{displaymath}
{\bf Step 3 (conclusion).} By the two previous steps, there exists a constant $\mathfrak c_1 > 0$, not depending on $m$, $N$ and $T$ (because $f = f_0$), such that
\begin{eqnarray*}
 \mathbb P(\Delta_{m}^{c})
 & \leqslant &
 \mathbb P(\Delta_{m,\mathfrak c}^{c}) +
 \mathbb P(\Delta_{m,\mathfrak l}^{c})\\
 & \leqslant &
 \mathfrak c_1(
 (\mathfrak L(m)^{\frac{1}{2}} +\mathfrak R(m)^{\frac{1}{2}})
 (N^{-\frac{1}{2}}\mathfrak I(m)^{\frac{1}{2}}T^{-1} +
 \mathfrak R(m)^{\frac{1}{2}}e^{\mathfrak cT}T^{2H - 1}) +
 \mathfrak R(m)e^{\mathfrak cT}T^{2H}).
\end{eqnarray*}
Moreover, by Assumption \ref{assumption_LRI}, if $T\in (0,1]$, then there exists a constant $\mathfrak c_2 > 0$, not depending on $m$, $N$ and $T$, such that
\begin{displaymath}
\mathbb P(\Delta_{m}^{c})
\leqslant\mathfrak c_2\mathfrak R(m)(
N^{-\frac{1}{2}}T^{-1} + T^{2H - 1}).
\end{displaymath}
\end{proof}
%


%
\begin{theorem}\label{risk_bound_fixed_point_estimator}
Under Assumption \ref{assumption_LRI}, if $T\in (0,1]$, $b_0$ fulfills the dissipativity condition (\ref{dissipativity_condition}), and if $X^1,\dots,X^N$ are independent copies of the stationary solution $X$ of Equation (\ref{main_equation}), then there exists a constant $\mathfrak c_{\ref{risk_bound_fixed_point_estimator}} > 0$, not depending on $m$, $N$ and $T$, such that
\begin{displaymath}
\mathbb E(\|\widetilde b_{m}^{\mathfrak c,\mathfrak l} - b_0\|_{f^2}^{2})
\leqslant
2\|b_m - b_0\|_{f^2}^{2} +
\mathfrak c_{\ref{risk_bound_fixed_point_estimator}}
\mathfrak R(m)V(N,T),
\end{displaymath}
where ($f = f_0$ doesn't depend on $T$, and)
\begin{displaymath}
V(N,T) := N^{-\frac{1}{2}}T^{-1} + T^{2H - 1}.
\end{displaymath}
\end{theorem}
%


%
\begin{proof}
The proof of Theorem \ref{risk_bound_fixed_point_estimator} is dissected in two steps. The first one deals with the following regularity result on the map $\Phi_m$:
\begin{equation}\label{risk_bound_fixed_point_estimator_1}
\|\Phi_m(\varphi) -\Phi_m(\overline\varphi)\|_{f^2}
\leqslant
\mathfrak l(m,T)
\|\varphi' -\overline\varphi'\|_{\infty}
\textrm{ $;$ }
\forall\varphi,\psi\in\mathcal S_{\mathfrak c}
\end{equation}
with
\begin{displaymath}
\mathcal S_{\mathfrak c} =\{\varphi\in C^1(\mathbb R) :\|\varphi'\|_{\infty}\leqslant\mathfrak c\}
\supset\mathcal S_{m,\mathfrak c}
\quad {\rm and}\quad
\mathfrak l(m,T) =
\frac{\mathfrak a}{2H(2H + 1)}
\mathfrak R(m)^{\frac{1}{2}}e^{\mathfrak cT}T^{2H}.
\end{displaymath}
Then, Step 2 deals with the $\mathbb L^2$-error bound on $\widetilde b_{m}^{\mathfrak c,\mathfrak l}$.
\\
\\
{\bf Step 1.} For any $\varphi,\overline\varphi\in\mathcal S_{\mathfrak c}$, since $(\varphi_1,\dots,\varphi_m)$ is an orthonormal family of $\mathbb L^2(\mathbb R,dx)$,
\begin{displaymath}
\|\Phi_m(\varphi) -\Phi_m(\overline\varphi)\|_{f^2} =
\|F(\varphi) - F(\overline\varphi)\|_m,
\end{displaymath}
where $F$ is the map from $\mathcal S_{\mathfrak c}$ into $\mathbb R^m$ defined by
\begin{displaymath}
F(\psi) :=\left(
\frac{\mathfrak a}{NT}\sum_{i = 1}^{N}\int_{0}^{T}\int_{0}^{t}
\varphi_j'(X_{t}^{i})\exp\left(\int_{s}^{t}\psi'(X_{u}^{i})du\right)|t - s|^{2H - 2}dsdt
\right)_{j\in\{1,\dots,m\}}
\end{displaymath}
for every $\psi\in\mathcal S_{\mathfrak c}$. Moreover, for every $j\in\{1,\dots,m\}$,
\begin{eqnarray*}
 |F(\varphi)_j - F(\overline\varphi)_j|
 & \leqslant &
 \frac{\mathfrak a}{NT}\sum_{i = 1}^{N}
 \int_{0}^{T}\int_{0}^{t}|\varphi_j'(X_{t}^{i})|\cdot |t - s|^{2H - 2}\\
 & &
 \hspace{2.5cm}\times
 \left|\exp\left(\int_{s}^{t}\varphi'(X_{u}^{i})du\right) -
 \exp\left(\int_{s}^{t}\overline\varphi'(X_{u}^{i})du\right)\right|dsdt\\
 & \leqslant &
 \frac{\mathfrak a}{NT}\|\varphi_j'\|_{\infty}e^{\mathfrak cT}\left(
 N\int_{0}^{T}\int_{0}^{t}|t - s|^{2H - 1}dsdt\right)\|\varphi' -\overline\varphi'\|_{\infty}\\
 & = &
 \frac{\mathfrak a}{2H(2H + 1)}
 \|\varphi_j'\|_{\infty}e^{\mathfrak cT}T^{2H}
 \|\varphi' -\overline\varphi'\|_{\infty}.
\end{eqnarray*}
This leads to Inequality (\ref{risk_bound_fixed_point_estimator_1}):
\begin{displaymath}
\|\Phi_m(\varphi) -\Phi_m(\overline\varphi)\|_{f^2}
\leqslant
\mathfrak c_1
\mathfrak R(m)^{\frac{1}{2}}e^{\mathfrak cT}T^{2H}
\|\varphi' -\overline\varphi'\|_{\infty}
\quad {\rm with}\quad
\mathfrak c_1 =
\frac{\mathfrak a}{2H(2H + 1)}.
\end{displaymath}
{\bf Step 2.} On $\Delta_m$, by Equality (\ref{relation_pLS_drift}) and Inequality (\ref{risk_bound_fixed_point_estimator_1}),
\begin{eqnarray*}
 \|\widetilde b_m - b_0\|_{f^2} & \leqslant &
 \|\Phi_m(\widetilde b_m) -\Phi_m(b_0)\|_{f^2} +
 \|\Phi_m(b_0) - b_0\|_{f^2}\\
 & \leqslant &
 \mathfrak l(m,T)\underbrace{\|\widetilde b_m' - b_0'\|_{\infty}}_{\leqslant 2\mathfrak c} +
 \|\widehat b_m - b_0\|_{f^2}.
\end{eqnarray*}
Moreover, by Proposition \ref{risk_bound_oracle},
\begin{eqnarray*}
 \mathbb E(\|\widehat b_m - b_0\|_{f^2}^{2}) & \leqslant &
 \|b_m - b_0\|_{f^2}^{2} +
 \frac{2}{N}\left(\|b_0\|_{f}^{2}\mathfrak L(m) +
 \mathfrak c_{\ref{bound_variance_Skorokhod}}\sigma^2
 \frac{\overline{\mathfrak m}_T}{T^{2 - 2H}}(\mathfrak L(m) +\mathfrak R(m))\right)\\
 & \leqslant &
 \|b_m - b_0\|_{f^2}^{2} +
 \frac{\mathfrak c_2}{N}(\mathfrak L(m) + T^{2H - 2}\mathfrak R(m))
\end{eqnarray*}
with
\begin{displaymath}
\mathfrak c_2 =
2\max\left\{
\|b_0\|_{f}^{2}\textrm{ $;$ }
\mathfrak c_{\ref{bound_variance_Skorokhod}}
\sigma^2\left[1\vee\left(-\frac{H}{M}\right)^{2H}\right]\right\}.
\end{displaymath}
So,
\begin{eqnarray*}
 \mathbb E(\|\widetilde b_{m}^{\mathfrak c,\mathfrak l} - b_0\|_{f^2}^{2})
 & = &
 \mathbb E(\|\widetilde b_m - b_0\|_{f^2}^{2}\mathbf 1_{\Delta_m}) +
 \|b_0\|_{f^2}^{2}\mathbb P(\Delta_{m}^{c})\\
 & \leqslant &
 8\mathfrak c^2\mathfrak l(m,T)^2
 + 2\|b_m - b_0\|_{f^2}^{2}
 + 2\mathfrak c_2N^{-1}(\mathfrak L(m) + T^{2H - 2}\mathfrak R(m))\\
 & &
 \hspace{5cm} +
 \mathfrak c_{\ref{bound_complement_Omega},2}
 \|b_0\|_{f^2}^{2}\mathfrak R(m)(N^{-\frac{1}{2}}T^{-1} + T^{2H - 1})\\
 & \leqslant &
 2\|b_m - b_0\|_{f^2}^{2} +
 \mathfrak c_3\mathfrak R(m)(N^{-\frac{1}{2}}T^{-1} + T^{2H - 1}),
\end{eqnarray*}
where $\mathfrak c_3$ is a positive constant not depending on $m$, $N$ and $T$ (because $f = f_0$).
\end{proof}
\noindent
{\bf Remarks:}
\begin{itemize}
 \item Consider $T(N) := N^{-1/(4H)}$. Then, for every $t\in (0,1]$,
 \begin{displaymath}
 V(N,T(N)) =
 2N^{-\frac{2H - 1}{4H}}\leqslant V(N,t),
 \end{displaymath}
 and by Theorem \ref{risk_bound_fixed_point_estimator} with $T = T(N)$,
 \begin{displaymath}
 \mathbb E(\|\widetilde b_{m}^{\mathfrak c,\mathfrak l} - b_0\|_{f^2}^{2})
 \leqslant
 2\|b_m - b_0\|_{f^2}^{2} +
 2\mathfrak c_{\ref{risk_bound_fixed_point_estimator}}\mathfrak R(m)N^{-\frac{2H - 1}{4H}}.
 \end{displaymath}
 \item Assume that $(\varphi_1,\dots,\varphi_m)$ is the trigonometric basis defined in Example \ref{example_assumption_LRI}, and that $(b_0f)_{|I}$ belongs to the Fourier-Sobolev space
 \begin{displaymath}
 \mathbb W_{2}^{\beta}(I) :=
 \left\{\varphi : I\rightarrow\mathbb R
 \textrm{ $\beta$ times differentiable on $I$} :
 \int_I\varphi^{(\beta)}(x)^2dx <\infty\right\}
 \end{displaymath}
 with $\beta\in\mathbb N^*$. By DeVore and Lorentz \cite{DL93}, Corollary 2.4 p. 205, there exists a constant $\mathfrak c_{\beta,\ell,\texttt r} > 0$, not depending on $m$, such that
 \begin{displaymath}
 \|(bf)_m - b_0f\|^2\leqslant
 \mathfrak c_{\beta,\ell,\texttt r}m^{-2\beta}.
 \end{displaymath}
 Then,
 \begin{displaymath}
 \mathbb E(\|\widetilde b_{m}^{\mathfrak c,\mathfrak l} - b_0\|_{f^2}^{2})
 \leqslant
 \overline{\mathfrak c}_{\ref{risk_bound_fixed_point_estimator}}
 (m^{-2\beta} + m^3N^{-\frac{2H - 1}{4H}}),
 \end{displaymath}
 where $\overline{\mathfrak c}_{\ref{risk_bound_fixed_point_estimator}} > 0$ is a positive constant not depending on $m$ and $N$. Therefore, the bias-variance tradeoff is reached for
 \begin{displaymath}
 m\asymp
 N^{\frac{2H - 1}{4H(3 + 2\beta)}}.
 \end{displaymath}
 \item Thanks to a well-known consequence of Picard's fixed point theorem, for every $n\in\mathbb N^*$ and $\varphi\in\mathcal S_{m,\mathfrak c}$,
 \begin{displaymath}
 \|\underbrace{(\Phi_m\circ\dots\circ\Phi_m)}_{n\textrm{ times}}(\varphi) -\widetilde b_m\|_{f^2}
 \leqslant
 \frac{\mathfrak l^n}{1 -\mathfrak l}
 \|\Phi_m(\varphi) -\varphi\|_{f^2}
 \quad {\rm on}\quad\Delta_m.
 \end{displaymath}
 Then, the estimator $\widetilde b_{m}^{\mathfrak c,\mathfrak l}$ can be approximated by $\widetilde b_{m,n}^{\mathfrak c,\mathfrak l} :=\widetilde b_{m,n}\mathbf 1_{\Lambda_m}$, where the sequence $(\widetilde b_{m,n})_{n\in\mathbb N}$ is defined by $\widetilde b_{m,0}\in\mathcal S_{m,\mathfrak c}$ and
 \begin{displaymath}
 \widetilde b_{m,n + 1} =
 \Phi_m(\widetilde b_{m,n})
 \textrm{ $;$ }n\in\mathbb N.
 \end{displaymath}
 \item In practice, the density function $f$ is unknown. So, $f$ needs to be estimated, for instance by the projection estimator
\begin{displaymath}
\widehat f_m(x) :=
\sum_{j = 1}^{m}\left(\frac{1}{NT}\sum_{i = 1}^{N}
\int_{0}^{T}\varphi_j(X_{s}^{i})ds\right)\varphi_j(x)
\textrm{ $;$ }x\in I.
\end{displaymath}
Then, for numerical purposes, $(\widetilde b_{m,n})_{n\in\mathbb N}$ should be replaced by the sequence $(\overline b_{m,n})_{n\in\mathbb N}$ such that $\overline b_{m,0}\in\mathcal S_{m,\mathfrak c}$ and
\begin{displaymath}
\overline b_{m,n + 1} =
\widehat\Phi_m(\overline b_{m,n})
\textrm{ $;$ }n\in\mathbb N,
\end{displaymath}
where
\begin{displaymath}
\widehat\Phi_m(.) :=\frac{f}{\widehat f_m}\Phi_m(.).
\end{displaymath}
\end{itemize}
%


%

%
\end{document}